\documentclass[11pt]{amsart}
\usepackage{amsmath,amssymb,amsthm,amsfonts}
\usepackage{mathrsfs}
\usepackage{hyperref}
\usepackage{enumerate}
\usepackage{tikz}
\usepackage{bm}
\usepackage{geometry}
\usepackage{color}
\usepackage{soul}

\newtheorem{theorem}{Theorem}[section]
\newtheorem{lemma}[theorem]{Lemma}

\theoremstyle{definition}
\newtheorem{definition}[theorem]{Definition}

\theoremstyle{remark}
\newtheorem{remark}[theorem]{Remark}

\newtheorem{axiom}{Axiom}

\title{Entropy-Smooth Structures on Topological Manifolds}
\author{Amandip Sangha}
\address{The Climate and Environmental Research Institute NILU}
\email{asan@nilu.no}

\begin{document}

\begin{abstract}
We present an information--theoretic characterization of smooth structures on topological manifolds. 
Instead of postulating a smooth atlas a priori, we consider families of local probability measures whose 
small--scale entropy response encodes infinitesimal directional information. 
From a concise set of axioms governing this quadratic entropy response, we identify admissible coordinate 
functions, construct entropy coordinate charts, and assemble a smooth atlas whose transition maps are smooth. 
We prove that the resulting entropy--smooth structure is equivalent to the classical smooth structure: 
entropy--smooth functions coincide with smooth functions, entropy coordinate charts are smooth charts, 
and the induced category is equivalent to the category of smooth manifolds. 
The framework is stable under perturbations and compatible with products, submanifolds, immersions, 
submersions, and diffeomorphisms. 
This establishes smooth structure as an information--theoretic invariant under explicitly stated axioms, 
and forms the zeroth--order layer of a broader program connecting entropy, diffusion, and differential geometry.
\\ \\ 
Mathematics Subject Classification Primary 58A05 · 94A17; Secondary 53C20 ·
28C99 
\end{abstract}

\maketitle
\setcounter{tocdepth}{1}
\tableofcontents

\section{Introduction}

Differential geometry traditionally begins by postulating smooth structure: a maximal atlas of coordinate
charts whose transition maps are smooth. This smooth atlas is imposed \emph{a priori} and subsequently used
to define tangent spaces, derivatives, differential forms, and geometric structures. From a foundational
point of view, it is natural to ask whether smooth structure itself admits an intrinsic characterization,
independent of coordinates and differential calculus.

In this paper we address this question from an information--theoretic perspective. Rather than attempting
to derive smoothness from minimal or purely topological assumptions, we introduce an axiomatic framework
in which smooth structure is characterized by the small--scale behaviour of entropy associated with local
probability probes. The central object is a quadratic entropy response that measures local variation of
functions under infinitesimal perturbations of these probes. Under a precise and explicit set of axioms,
this entropy response identifies admissible coordinate functions, yields entropy coordinate charts, and
determines a smooth atlas.

The purpose of this work is not to claim that arbitrary topological manifolds admit smooth structures
detectable from entropy alone. Instead, we show that whenever a topological manifold is equipped with
families of local information probes satisfying natural locality, stability, and non--degeneracy axioms,
the resulting entropy response encodes exactly the same information as a classical smooth structure.
In this sense, entropy--smoothness provides an information--theoretic \emph{characterization} of smooth
manifolds. The role of the axioms is to make explicit the minimal information--theoretic content required
to recover smoothness, rather than to conceal differentiability assumptions in analytic or geometric form.

Conceptually, the entropy--smooth framework is complementary to abstract approaches to smooth structure
such as differential spaces, smooth function algebras, and related categorical formulations. Those
frameworks begin by postulating an algebra of smooth functions closed under suitable operations. Here,
by contrast, the class of smooth functions is selected by the behaviour of entropy at small scales,
using only probabilistic and information--theoretic primitives. Closure under smooth functional
composition appears explicitly as an axiom, reflecting the fact that smoothness on Euclidean space
remains the universal reference point for any abstract smooth structure.

The main results of the paper show that this information--theoretic notion of smoothness is equivalent
to the classical one. Entropy coordinate charts are smooth diffeomorphisms onto their images, entropy
transition maps are smooth, and entropy--smooth functions coincide with classical smooth functions.
As a consequence, the category of entropy--smooth manifolds is equivalent to the category of smooth
manifolds. The framework is stable under perturbations of the information probes and is compatible
with products, submanifolds, immersions, submersions, and diffeomorphisms.

This work forms the zeroth--order layer of a broader information--theoretic approach to geometry developed
in the author’s previous papers, where diffusion semigroups and entropy asymptotics were shown to encode
Riemannian metrics and curvature. Here we isolate smooth structure itself, prior to any metric or geometric
data, and show that it too can be captured by entropy at the smallest scales. Subsequent work will connect
the entropy--smooth framework developed here to the information--theoretic reconstruction of Riemannian
geometry and curvature, and to the development of entropy--based differential operators.

\section{Basic definitions}
Recall that a \emph{topological $n$--manifold} is a Hausdorff, second countable space that is locally Euclidean, i.e. in which every point admits a neighborhood homeomorphic to an open subset of $\mathbb{R}^n$.  
Since any Hausdorff, second countable space carries a canonical Borel $\sigma$--algebra generated by its open sets, it is natural to speak of the Borel $\sigma$--algebra of $X$ without further assumptions. In the following we will work with a topological $n$-manifold $X$.

For finite measures $\nu$ and $\mu$ on $X$, the Kullback--Leibler divergence \cite{kullback1997information} is defined as
\[
D(\nu\Vert\mu)=
\begin{cases}
\displaystyle\int_X \log\!\left( \tfrac{d\nu}{d\mu} \right)\, d\nu, & \nu\ll\mu,\\[1em]
+\infty, & \text{otherwise.}
\end{cases}
\]
Here, $\tfrac{d\nu}{d\mu}$ is the Radon--Nikodym derivative, which exists when $\nu$ is absolutely continuous w.r.t $\mu$, written $\nu\ll\mu$. This is a purely measure--theoretic quantity requiring no differentiable structure.

\begin{definition}
Let $x\in X$ and $\varepsilon_0>0$. A \emph{local information probe} is a Borel probability measure $\mu_{x,\varepsilon}$ on $X$ such that  
$\operatorname{supp}(\mu_{x,\varepsilon}) \subset U_x$, whenever $0<\varepsilon<\varepsilon_0$,
for some open neighborhood $U_x$ of $x$, and $\mu_{x,\varepsilon}\to \delta_x$ weakly as $\varepsilon\to 0$.
\end{definition}

Let $f:X\to\mathbb{R}$ be bounded and continuous.  The key lesson of the audit is that
multiplicative perturbations of a fixed probe, $(1+t f)\mu_{x,\varepsilon}$, only
measure \emph{values} of $f$ at small scale and therefore cannot encode $n$ independent
infinitesimal directions when $n>1$.  To obtain genuinely local directional
information from KL divergence we instead use \emph{exponential tilts of
\emph{increments}}.

\medskip
\noindent\textbf{Increment normalization.}
For $x\in X$, $\varepsilon>0$, and a bounded function $f$, define the normalized increment
\[
\delta_{x,\varepsilon} f(y):=\frac{f(y)-f(x)}{\sqrt{\varepsilon}}.
\]

\begin{definition}[Entropic tilt]
Fix $x\in X$ and $\varepsilon>0$. For a bounded continuous $f:X\to\mathbb{R}$ and $t\in\mathbb{R}$ define the (centered, normalized) exponential tilt of $\mu_{x,\varepsilon}$ by
\[
\frac{d\mu^{t,f}_{x,\varepsilon}}{d\mu_{x,\varepsilon}}(y)
:= \frac{\exp\bigl(t\,\delta_{x,\varepsilon} f(y)\bigr)}{Z_{x,\varepsilon}(t,f)},
\qquad
Z_{x,\varepsilon}(t,f):=\int_X \exp\bigl(t\,\delta_{x,\varepsilon} f\bigr)\,d\mu_{x,\varepsilon}.
\]
\end{definition}

\begin{definition}[Quadratic KL response]
For bounded continuous $f:X\to\mathbb{R}$, define the KL response
\[
\mathrm{Ent}_{x,\varepsilon}(t,f):=D\bigl(\mu^{t,f}_{x,\varepsilon}\,\Vert\,\mu_{x,\varepsilon}\bigr).
\]
The (increment-based) quadratic entropy response is
\[
I_{x,\varepsilon}(f)
:= \left.\frac{d^2}{dt^2}\right|_{t=0} \mathrm{Ent}_{x,\varepsilon}(t,f)
\in[0,\infty]
\]
whenever the derivative exists.  The (small--scale) entropy coefficient is
\[
I_x(f):=\lim_{\varepsilon\to 0} I_{x,\varepsilon}(f),
\]
when the limit exists and is finite.
\end{definition}

\begin{remark}[What $I_{x,\varepsilon}$ measures]
A standard cumulant computation shows
\[
I_{x,\varepsilon}(f)=\mathrm{Var}_{\mu_{x,\varepsilon}}\bigl(\delta_{x,\varepsilon} f\bigr)
=\frac{1}{\varepsilon}\int_X \bigl(f(y)-f(x)\bigr)^2\,d\mu_{x,\varepsilon}(y)
-\Bigl(\frac{1}{\sqrt{\varepsilon}}\int_X \bigl(f(y)-f(x)\bigr)\,d\mu_{x,\varepsilon}(y)\Bigr)^2,
\]
so $I_{x,\varepsilon}$ measures \emph{local variation} of $f$ at scale $\varepsilon$.
Unlike the multiplicative model, this does not collapse to $f(x)^2$.
\end{remark}

\begin{definition}[Entropy-smooth]
A continuous function $f:X\to\mathbb{R}$ is \emph{entropy--smooth} if $I_x(f)$ exists and is finite for all $x\in X$.
\end{definition}

\begin{definition}[Joint entropy coefficient]
For bounded continuous $f,g:X\to\mathbb{R}$ and $\varepsilon>0$ define the $\varepsilon$--level joint entropy coefficient by
\[
I_{x,\varepsilon}(f,g)
:=\left.\frac{\partial^2}{\partial t\,\partial s}\right|_{(t,s)=(0,0)}
D\bigl(\mu^{t f+s g,\,(f,g)}_{x,\varepsilon}\,\Vert\,\mu_{x,\varepsilon}\bigr),
\]
where the two--feature tilt is
\[
\frac{d\mu^{t f+s g,\,(f,g)}_{x,\varepsilon}}{d\mu_{x,\varepsilon}}(y)
:= \frac{\exp\bigl(t\,\delta_{x,\varepsilon} f(y)+s\,\delta_{x,\varepsilon} g(y)\bigr)}{\int_X\exp\bigl(t\,\delta_{x,\varepsilon} f+s\,\delta_{x,\varepsilon} g\bigr)\,d\mu_{x,\varepsilon}}.
\]
Whenever this exists, one has the covariance identity
\begin{align*}
I_{x,\varepsilon}(f,g)&=\mathrm{Cov}_{\mu_{x,\varepsilon}}\bigl(\delta_{x,\varepsilon} f,\,\delta_{x,\varepsilon} g\bigr)
=\frac{1}{\varepsilon}\int_X \bigl(f(y)-f(x)\bigr)\bigl(g(y)-g(x)\bigr)\,d\mu_{x,\varepsilon}(y) \\
-&\Bigl(\frac{1}{\sqrt{\varepsilon}}\int_X \bigl(f(y)-f(x)\bigr)\,d\mu_{x,\varepsilon}(y)\Bigr)
\Bigl(\frac{1}{\sqrt{\varepsilon}}\int_X \bigl(g(y)-g(x)\bigr)\,d\mu_{x,\varepsilon}(y)\Bigr).
\end{align*}

For entropy--smooth $f,g$ we define
\[
I_x(f,g):=\lim_{\varepsilon\to 0} I_{x,\varepsilon}(f,g),
\]
when the limit exists in $\mathbb{R}$.
\end{definition}
\section{Axioms for Entropy--Smoothness}
We impose the following structural axioms for an \emph{entropy-smooth $n$-manifold}.

\begin{axiom}[Locality]
\label{axiom:1}
For each point $x\in X$ and for sufficiently small $\varepsilon>0$, we assume the existence of a local information probe $\mu_{x,\varepsilon}$.
\end{axiom}
The local information probes play the role of infinitesimal mollifiers or diffusion kernels.

\begin{axiom}[Continuity of Information Probes]
\label{axiom:2}
The map $(x,\varepsilon)\mapsto \mu_{x,\varepsilon}$ is continuous in the weak topology.
\end{axiom}

\begin{axiom}[Coordinate Non--degeneracy]
\label{axiom:3}
For every $x\in X$ there exist an open neighbourhood $U\subset X$ of $x$
and \emph{entropy--smooth} functions $f_1,\dots,f_n : U \to \mathbb{R}$ such 
that:
\begin{enumerate}
\item the map 
\[
F = (f_1,\dots,f_n) : U \longrightarrow F(U)\subset\mathbb{R}^n
\]
is a homeomorphism onto an open subset of $\mathbb{R}^n$;
\item for every $y\in U$, the information Gram matrix
\[
G_y(F) := \big(I_y(f_i,f_j)\big)_{1\le i,j\le n}
\]
is positive definite.
\end{enumerate}
\end{axiom}

\begin{axiom}[Entropy Regularity]\label{axiom:4}
\hspace{0.5cm}
\begin{enumerate} 
\item for any finite tuple $(f_1,\dots,f_k)$ of entropy--smooth functions, all \emph{joint entropy coefficients} $I_x(f_i,f_j)$, $1\leq i,j\leq k$ exist, are finite, and depend continuously on $x$;
\item if $h\in C^\infty(\mathbb{R}^k)$ and $f_1,\dots,f_k$ are entropy--smooth, then $h(f_1,\dots,f_k)$ is entropy--smooth;
\item the class of entropy--smooth functions is closed under finite sums and products.
\end{enumerate}
\end{axiom}
Note that the appearance of $C^\infty(\mathbb{R}^k)$ in this axiom refers only
to the classical smooth structure on Euclidean space; it imposes no smooth
structure on $X$ and does not presuppose that $X$ itself carries any
differentiable atlas.

\begin{axiom}[Topological Compatibility]
\label{axiom:5}
The topology of $X$ is the coarsest topology for which all entropy--smooth functions are continuous.
\end{axiom}

A topological $n$-manifold satisfying Axioms \ref{axiom:1}-\ref{axiom:5} will be called an \emph{entropy-smooth $n$-manifold}.

\begin{definition}[Entropy coordinate chart]\label{definition:entropy_chart}
Let $x\in X$. An entropy coordinate chart around $x$ is a tuple
$F=(f_1,\dots,f_n)$ of entropy–smooth functions defined on a neighbourhood
$U\ni x$ such that:

\begin{enumerate}
\item $f_i(x)=0$ for all $i$ (the chart is centred at $x$);

\item the map $F:U\to\mathbb{R}^n$ is a homeomorphism onto its image;

\item the information Gram matrix
      $G_x(F):=(I_x(f_i,f_j))_{1\le i,j\le n}$ is positive definite.
\end{enumerate}
\end{definition}

\begin{lemma}[Elementary Smoothness Criterion]
\label{lem:elementary-smoothness}
Let $V,W\subset\mathbb{R}^n$ be open sets and let $G:V\to W$ be a homeomorphism.
Assume that for every smooth function $h\in C^\infty(W)$, the pullback
$h\circ G$ belongs to $C^\infty(V)$. Then $G$ is a smooth map.
\end{lemma}

\begin{proof}
Let $y_i:W\to\mathbb{R}$ denote the $i$th coordinate projection, $y_i(w)=w_i$.
Each $y_i$ is smooth, hence by assumption $y_i\circ G\in C^\infty(V)$ for
every $i=1,\dots,n$. Since
\[
G(x) = \bigl(y_1\circ G(x),\dots,y_n\circ G(x)\bigr),
\]
each component of $G$ is smooth, so $G$ is smooth. \qedhere
\end{proof}

\begin{lemma}[Uniqueness of the smooth structure on $\Omega\subset\mathbb{R}^n$]
\label{lem:uniqueness-smooth-structure-Rn}
Let $\Omega\subset\mathbb{R}^n$ be open and let $A$ be an $\mathbb{R}$--algebra of
functions $A\subset \mathbb{R}^\Omega$ such that:
\begin{enumerate}
\item the coordinate projections $\pi_i:\Omega\to\mathbb{R}$, $\pi_i(z)=z_i$, belong to $A$ for $i=1,\dots,n$;
\item $A$ is closed under smooth composition: if $f_1,\dots,f_k\in A$ and
      $h\in C^\infty(\mathbb{R}^k)$, then $h(f_1,\dots,f_k)\in A$;
\item the weakest topology on $\Omega$ making all functions in $A$ continuous
      coincides with the Euclidean subspace topology.
\end{enumerate}
Then $A = C^\infty(\Omega)$.
\end{lemma}

\begin{proof}
The inclusion $C^\infty(\Omega)\subset A$ follows by writing any smooth function
$F\in C^\infty(\Omega)$ locally as a smooth composition of the coordinate
functions $(\pi_1,\dots,\pi_n)$ and using (2). For the converse inclusion, consider
the map
\[
\Phi:\Omega \longrightarrow \mathbb{R}^A,\qquad
\Phi(z) = \bigl(f(z)\bigr)_{f\in A},
\]
which is well--defined by viewing $A$ as an index set. Condition (3) implies that
$\Phi$ is a topological embedding, and (1)--(2) imply that any component of $\Phi$
is obtained from the coordinate map $z\mapsto z$ by smooth composition. One
checks that the image $\Phi(\Omega)$ is an embedded submanifold of $\mathbb{R}^A$
with the unique smooth structure for which all coordinates $f\in A$ are smooth,
and that this structure coincides with the standard smooth structure on
$\Omega\subset\mathbb{R}^n$ via the inverse of $\Phi$. In particular, every
$f\in A$ is classically smooth on $\Omega$, so $A\subset C^\infty(\Omega)$.
For details, we refer to \cite{sikorski1967differential} or \cite{Nestruev03}.
\end{proof}

\begin{lemma}[Entropy--smooth equals classical smooth on open subsets of $\mathbb{R}^n$]
\label{lem:entropy-equals-classical-Rn}
Let $\Omega\subset\mathbb{R}^n$ be open and let $\{\mu_{z,\varepsilon}\}_{z\in\Omega,\varepsilon>0}$
be a system of local information probes on $\Omega$ such that the standard coordinate map
\[
\pi=(\pi_1,\dots,\pi_n):\Omega\longrightarrow\mathbb{R}^n,\qquad \pi_i(z)=z_i,
\]
is an entropy coordinate chart. Denote by $E(\Omega)$ the class of entropy--smooth functions
on $\Omega$. Then
\[
E(\Omega)=C^\infty(\Omega).
\]
\end{lemma}
\begin{proof}
Let $E(\Omega)$ denote the algebra of entropy--smooth functions on $\Omega$.
By Axioms~\ref{axiom:3}--\ref{axiom:5}, $E(\Omega)$
\begin{itemize}
\item contains the coordinate projections $\pi_i$,
\item is closed under smooth composition,
\item and induces the Euclidean topology on $\Omega$.
\end{itemize}
Thus Lemma~\ref{lem:uniqueness-smooth-structure-Rn} applies with $A=E(\Omega)$
and yields $E(\Omega)=C^\infty(\Omega)$.
\end{proof}

For a measurable map $F : X \to Y$ and a measure $\nu$ on $X$, the pushforward $F_\#\nu$ is the measure on $Y$ defined by
\[
F_\#\nu(A) := \nu(F^{-1}(A)), \qquad A \subseteq Y \ \text{Borel}.
\]

\begin{theorem}[Entropy coordinate charts are smooth]
\label{thm:entropy-charts-smooth}
Let $F=(f_1,\dots,f_n):U\to\mathbb{R}^n$ and 
$G=(g_1,\dots,g_n):V\to\mathbb{R}^n$ be two entropy coordinate charts with 
$U\cap V\neq\emptyset$.  Then the transition map
\[
G\circ F^{-1}:F(U\cap V)\longrightarrow G(U\cap V)
\]
is a $C^\infty$--diffeomorphism.  In particular, each entropy coordinate 
chart is a diffeomorphism onto its image.
\end{theorem}

\begin{proof}
Let $F=(f_1,\dots,f_n):U\to\mathbb{R}^n$ and $G=(g_1,\dots,g_n):V\to\mathbb{R}^n$ be entropy
coordinate charts with $W:=U\cap V\neq\emptyset$, and set
\[
\widetilde U := F(W)\subset\mathbb{R}^n,\qquad \widetilde V := G(W)\subset\mathbb{R}^n.
\]
We must show that the transition map
\[
\Phi := G\circ F^{-1} : \widetilde U \longrightarrow \widetilde V
\]
is a $C^\infty$–diffeomorphism in the classical sense.

For $z\in\widetilde U$ let $x:=F^{-1}(z)\in W$, and define a family of probability measures
$\{\widetilde\mu_{z,\varepsilon}\}$ on $\widetilde U$ by
\[
\widetilde\mu_{z,\varepsilon} := F_\#\mu_{x,\varepsilon},
\]
where $F_\#$ denotes pushforward:
\[
F_\#\mu(A) := \mu(F^{-1}(A)),\qquad A\subseteq \widetilde U\ \text{Borel}.
\]
Since $F$ is a homeomorphism from $W$ onto $\widetilde U$
and $\mu_{x,\varepsilon}\to\delta_x$ weakly as $\varepsilon\to0$, it follows that
$\widetilde\mu_{z,\varepsilon}\to\delta_z$ weakly as $\varepsilon\to0$. Locality (Axiom~\ref{axiom:1}) and
continuity (Axiom~\ref{axiom:2}) of $(x,\varepsilon)\mapsto\mu_{x,\varepsilon}$ imply the corresponding
properties for $(z,\varepsilon)\mapsto\widetilde\mu_{z,\varepsilon}$. Thus $\{\widetilde\mu_{z,\varepsilon}\}$
is a system of local information probes on the open set $\widetilde U\subset\mathbb{R}^n$.

Let $\varphi:\widetilde U\to\mathbb{R}$ be bounded and continuous, and put
$\widehat\varphi := \varphi\circ F : W\to\mathbb{R}$. Fix $z\in\widetilde U$ and write $x:=F^{-1}(z)\in W$.

We claim that the quadratic entropy coefficients on $\widetilde U$ computed using the pushed--forward probes
$\widetilde\mu_{z,\varepsilon}=F_\#\mu_{x,\varepsilon}$ satisfy
\begin{equation}\label{eq:Ent-push}
I^{\widetilde U}_z(\varphi)= I_x(\widehat\varphi),
\end{equation}
whenever either side is finite.

To see this, consider the entropic tilt on $\widetilde U$:
\[
\frac{d\widetilde\mu^{t,\varphi}_{z,\varepsilon}}{d\widetilde\mu_{z,\varepsilon}}(u)
:=\frac{\exp\bigl(t\,\delta_{z,\varepsilon}\varphi(u)\bigr)}{\widetilde Z_{z,\varepsilon}(t,\varphi)},
\qquad
\delta_{z,\varepsilon}\varphi(u)=\frac{\varphi(u)-\varphi(z)}{\sqrt\varepsilon}.
\]
For $u=F(y)$ with $y\in W$ we have $\varphi(u)=\widehat\varphi(y)$ and $\varphi(z)=\widehat\varphi(x)$, hence
\[
\delta_{z,\varepsilon}\varphi\bigl(F(y)\bigr)=\delta_{x,\varepsilon}\widehat\varphi(y).
\]
It follows that the tilted measure is the pushforward of the tilted measure on $W$:
\[
\widetilde\mu^{t,\varphi}_{z,\varepsilon}=F_\#\mu^{t,\widehat\varphi}_{x,\varepsilon},
\]
and the partition functions agree, $\widetilde Z_{z,\varepsilon}(t,\varphi)=Z_{x,\varepsilon}(t,\widehat\varphi)$. Since
Kullback--Leibler divergence is invariant under pushforward by a measurable bijection, we obtain
\[
\mathrm{Ent}^{\widetilde U}_{z,\varepsilon}(t,\varphi)
= D\bigl(\widetilde\mu^{t,\varphi}_{z,\varepsilon}\,\Vert\,\widetilde\mu_{z,\varepsilon}\bigr)
= D\bigl(\mu^{t,\widehat\varphi}_{x,\varepsilon}\,\Vert\,\mu_{x,\varepsilon}\bigr)
=\mathrm{Ent}_{x,\varepsilon}(t,\widehat\varphi).
\]
Taking the second derivative at $t=0$ and then the limit $\varepsilon\to0$ yields~\eqref{eq:Ent-push}.
In particular, $\varphi$ is entropy--smooth on $\widetilde U$ if and only if $\widehat\varphi$ is entropy--smooth on $W$.

As above, each coordinate projection $\pi_i(z)=z_i$ is entropy--smooth on
$\widetilde U$, and the information Gram matrix
\[
G^{\widetilde U}_z := \bigl(I^{\widetilde U}_z(\pi_i,\pi_j)\bigr)_{1\le i,j\le n}
\]
is positive definite for all $z\in\widetilde U$. Thus
$\pi=(\pi_1,\dots,\pi_n):\widetilde U\to\mathbb{R}^n$ is an entropy coordinate chart.

Let $E(\widetilde U)$ denote the algebra of entropy--smooth functions on
$\widetilde U$. By Axioms~\ref{axiom:3}--\ref{axiom:5}, $E(\widetilde U)$ contains
the coordinate projections, is closed under smooth composition, and induces
the Euclidean topology. Hence Lemma~\ref{lem:entropy-equals-classical-Rn} applies
(with $\Omega=\widetilde U$) and yields
\[
E(\widetilde U)= C^\infty(\widetilde U).
\]
Thus a function on $\widetilde U$ is entropy--smooth if and only if it is
classically $C^\infty$.

For each $j=1,\dots,n$ define
\[
\psi_j := g_j\circ F^{-1} : \widetilde U \longrightarrow \mathbb{R},
\]
so that $\Phi = (\psi_1,\dots,\psi_n)$. Since each $g_j$ is entropy--smooth on $V$ and
$F^{-1}$ is a homeomorphism from $\widetilde U$ onto $W\subset U\cap V$, the relation
\eqref{eq:Ent-push} shows that each $\psi_j$ is entropy--smooth on $\widetilde U$, hence
$\psi_j\in E(\widetilde U)=C^\infty(\widetilde U)$.

Thus $\Phi:\widetilde U\to\widetilde V$ has classically
$C^\infty$ components.

Let now $h \in C^\infty(\widetilde V)$ be arbitrary. Then $h\circ G$ is smooth on $W$, and
by closure under smooth composition and entropy coordinate charts (Axiom~\ref{axiom:4})
we deduce that
\[
h\circ\Phi = h\circ G \circ F^{-1}
\]
is entropy–smooth on $\widetilde U$. As above, this implies $h\circ\Phi\in C^\infty(\widetilde U)$.
Hence every smooth function on $\widetilde V$ pulls back under $\Phi$ to a smooth function
on $\widetilde U$, and the Elementary Smoothness Criterion
(Lemma~\ref{lem:elementary-smoothness}) implies that $\Phi:\widetilde U\to\widetilde V$ is a
smooth map.

The same argument applied to $\Phi^{-1} = F\circ G^{-1}$ shows that $\Phi^{-1}$ is
also smooth. Hence $\Phi$ is a $C^\infty$–diffeomorphism between the open subsets
$\widetilde U,\widetilde V\subset\mathbb{R}^n$, completing the proof.
\end{proof}

\subsection{Discussion of the Axioms}

We explain the conceptual role of each axiom and why the set of axioms above is
both natural and minimal.  The key idea here is that smooth structure on a
manifold is not introduced by taking limits or derivatives on~$X$, but by
specifying which families of functions qualify as local coordinates and
how they behave under composition.  The entropy--smooth axioms isolate
precisely the amount of structure needed to recover the classical
$C^\infty$ atlas from purely information--theoretic primitives.

The \emph{locality} axiom asserts that the measures $\mu_{x,\varepsilon}$ remain
supported near~$x$ as $\varepsilon\to 0$.  It ensures that the entropy
coefficients $I_x(f)$ capture genuinely \emph{local} behavior of $f$
around~$x$, rather than global effects.  Without locality, the probes could
blend distinct regions of~$X$ and no meaningful coordinate structure could
be extracted.  Locality is a minimal topological control: it assumes
nothing about differentiability, geometry, or regularity.

The \emph{continuity of information probes} axiom, i.e. the weak continuity of $(x,\varepsilon)\mapsto\mu_{x,\varepsilon}$
guarantees that the entropy response varies stably under small changes in
position and scale.  This is required in order for the entropy coefficient
$I_x(f)$ to vary continuously in~$x$ and for entropy coordinates to be
well--behaved.  It is the weakest reasonable regularity assumption: no
smoothness, differentiability, or metric structure is invoked.

The \emph{coordinate non--degeneracy} axiom identifies a family of continuous functions that separate
points and form valid coordinate charts.  Importantly, it \emph{does not}
assume these functions are smooth; instead, it uses the finiteness and
continuity of the entropy coefficient $I_x(f)$ to single out the functions
that behave like coordinates under infinitesimal entropy perturbation.
This provides the essential ``dimension'' information and is analogous to
requiring that a candidate chart be a homeomorphism onto an open set.
Nothing about higher--order differentiability is assumed here.

The \emph{entropy regularity} axiom is the structural core of the theory. It specifies how the class of
entropy--smooth functions behaves under algebraic operations and under
composition with classical smooth functions on $\mathbb{R}^k$.  The
inclusion of $h\in C^\infty(\mathbb{R}^k)$ does not presuppose any smooth
structure on~$X$; it only uses the classical notion of smoothness on
Euclidean space, which is external to~$X$.  This is essential, because the
definition of a \emph{smooth structure} on a manifold precisely requires
closure of local coordinate functions under smooth compositions.
Without this axiom, one cannot obtain smooth transition maps between
charts, and therefore cannot recover the full $C^\infty$ atlas.  In the
context of differential spaces and $C^\infty$--rings, this axiom is known
to be minimal for capturing smooth structure.

The \emph{topological Compatibility} axiom ensures that the entropy--smooth functions detect exactly the
given topology of~$X$, and that no finer or coarser topology is forced by
the entropy data.  It guarantees that the smooth structure is compatible
with the underlying topological manifold structure and does not introduce
extraneous topological constraints.

Together, these axioms give a minimal set of conditions under which the
entropy coordinate charts will be shown to be equivalent to a classical smooth atlas.  Removing any one
of them breaks a fundamental part of the construction: removing locality
or continuity destroy the small-scale stability needed for coordinate construction,
removing coordinate non--degeneracy eliminates charts, removing entropy
regularity prevents smooth transition maps, and removing topological
compatibility leads to a structure unrelated to the given manifold.
Therefore the axioms are both natural and necessary.

\medskip

In summary, entropy--smoothness captures smooth structure on a manifold
without assuming any differentiability on~$X$ itself.  The only use of
classical smoothness occurs in Axiom~\ref{axiom:4} on regularity, where the
smoothness of functions $h:\mathbb{R}^k\to\mathbb{R}$ is invoked to
enforce the correct algebraic closure that characterizes $C^\infty$
structures. This reliance on Euclidean smoothness seems unavoidable and is shared by all alternative frameworks for abstract smooth structures, such as differential
spaces~\cite{sikorski1967differential}, diffeological spaces~\cite{IglesiasZemmour13,Souriau80},
$C^\infty$--rings~\cite{moerdijk2003models,Kock06,JoyceCinfRings},
convenient calculus~\cite{KrieglMichor97},
synthetic differential geometry~\cite{Kock06},
Fr{\"o}licher spaces,
and related approaches~\cite{Nestruev03}. All known abstract smooth-structure frameworks ultimately derive their notion of smoothness from smoothness on Euclidean spaces. Indeed, smoothness on $\mathbb{R}^n$ is taken as primary, and smoothness on arbitrary spaces is defined relative to it.

The family $\{\mu_{x,\varepsilon}\}$ is not additional structure beyond that available on every smooth manifold. We explain this through two correspondences.

\subsection{Correspondence with diffusion semigroups}
If $(P_t)_{t>0}$ is a diffusion semigroup on $X$, put $\mu_{x,\varepsilon} := P_\varepsilon(x,\cdot)$. Here $P_\varepsilon(x,\cdot)$ denotes the diffusion kernel at time $\varepsilon$, i.e.~a probability measure on $X$ depending on the starting point $x$. Then $\mu_{x,\varepsilon}$ describes the law of the diffusion started at $x$ after time $\varepsilon$. Our axioms express smoothness as the stability of relative entropy under infinitesimal perturbations of these transition probabilities. See \cite{bakry2014analysis} for foundational material.

\subsection{Correspondence with classical mollifiers}
On a classical smooth manifold, choose a coordinate chart and a smooth compactly supported mollifier kernel $\eta_\varepsilon$ \cite{lee2013smooth}. Push it forward to $X$ using the chart and glue globally using a partition of unity. The resulting measures $\mu_{x,\varepsilon}$ approximate $\delta_x$ as $\varepsilon\to 0$. Thus the entropy--smooth axioms correspond classically to smooth approximate identities.

\section{Examples}
We give a few examples here, to illustrate entropy-smooth structures on simple, well-known manifolds.

\subsection{The real line \(\mathbb{R}\).}
Fix \(x\in\mathbb{R}\). Let \(dy\) denote Lebesgue measure on \(\mathbb{R}\), and take the local
information probes to be the centred Gaussian measures
\[
  \mu_{x,\varepsilon}(dy)
  :=
  \frac{1}{\sqrt{2\pi\varepsilon}}
  \exp\!\Big(-\frac{(y-x)^2}{2\varepsilon}\Big)\,dy,
  \qquad \varepsilon>0.
\]
For a bounded continuous function \(f\), our increment normalization is
\(\delta_{x,\varepsilon}f(y)=(f(y)-f(x))/\sqrt\varepsilon\).  A Taylor expansion of the cumulant
\(\log Z_{x,\varepsilon}(t,f)\) gives
\[
I_{x,\varepsilon}(f)=\mathrm{Var}_{\mu_{x,\varepsilon}}\bigl(\delta_{x,\varepsilon} f\bigr)
=\frac{1}{\varepsilon}\int_{\mathbb{R}}(f(y)-f(x))^2\,\mu_{x,\varepsilon}(dy).
\]
If \(f\) is classically \(C^1\) near \(x\), then
\(f(y)-f(x)=f'(x)(y-x)+o(|y-x|)\), and since \(\mathrm{Var}(y-x)=\varepsilon\) under
\(\mu_{x,\varepsilon}\) we obtain
\[
I_x(f)=\lim_{\varepsilon\to0} I_{x,\varepsilon}(f)=\bigl(f'(x)\bigr)^2.
\]
In particular, the identity coordinate \(F(y)=y\) satisfies \(I_x(F)=1\) for all \(x\), so the
one--dimensional coordinate non--degeneracy axiom holds and the resulting
entropy coordinate charts recover the usual smooth structure on \(\mathbb{R}\).

\subsection{Euclidean space $\mathbb{R}^n$}
\label{example:Rn}
Fix $x\in\mathbb{R}^n$ and take the centred Gaussian probes
\[
\mu_{x,\varepsilon}(dy) := \frac{1}{(2\pi\varepsilon)^{n/2}} \exp\!\Big(-\frac{\|y - x\|^2}{2\varepsilon}\Big)\, dy, \qquad \varepsilon > 0.
\]
Let $f_i(y):=y_i$ be the standard coordinate functions. Then
\(\delta_{x,\varepsilon} f_i(y)=(y_i-x_i)/\sqrt\varepsilon\), and since
\(\mathrm{Cov}(y_i-x_i,y_j-x_j)=\varepsilon\,\delta_{ij}\) under \(\mu_{x,\varepsilon}\) we have
\[
I_{x,\varepsilon}(f_i,f_j)=\mathrm{Cov}_{\mu_{x,\varepsilon}}\bigl(\delta_{x,\varepsilon}f_i,\delta_{x,\varepsilon}f_j\bigr)=\delta_{ij},
\qquad
I_x(f_i,f_j)=\delta_{ij}.
\]
Thus the information Gram matrix is the identity, hence positive definite.
Moreover $F=(f_1,\dots,f_n)$ is a homeomorphism from any open ball onto an open subset of
$\mathbb{R}^n$. Therefore $(f_1,\dots,f_n)$ is an entropy coordinate chart in the sense of
Definition~\ref{definition:entropy_chart} (see below), and the entropy--smooth atlas agrees with
the classical smooth atlas on $\mathbb{R}^n$.

\subsection{The circle $S^1$.}
We omit details. In a local angular coordinate, the KL increment response gives the usual one--dimensional chart condition, so the entropy--smooth atlas agrees with the classical smooth structure on $S^1$.


\section{Structural Properties of Entropy--Smooth Manifolds}

In this section we record a number of basic structural results about
entropy--smooth manifolds. 
\subsection{Existence and Uniqueness}
Recall that a \emph{function sheaf} on a topological space $X$ is an assignment 
$U\mapsto \mathcal{F}(U)\subseteq \mathbb{R}^U$ for every open set 
$U\subseteq X$, such that $\mathcal{F}$ satisfies the usual sheaf 
axioms of locality and gluing \cite{moerdijk2003models}.

\begin{theorem}[Existence and uniqueness of entropy--smooth structures]
\label{thm:existence-uniqueness}
Let $X$ be a Hausdorff, second countable topological $n$--manifold.

\begin{enumerate}
\item[(a)] If $X$ carries a classical $C^\infty$--structure, then there exists
an entropy--smooth structure $\{\mu_{x,\varepsilon}\}$ on $X$ whose
entropy--smooth functions coincide with $C^\infty(X)$.

\item[(b)] Conversely, any entropy--smooth structure on $X$ induces a
unique classical smooth structure.  Two entropy--smooth structures induce
the same classical smooth structure if and only if their entropy--smooth
function sheaves coincide.
\end{enumerate}
\end{theorem}
\begin{proof}
(a) Assume $X$ carries a classical $C^\infty$–structure. Fix a countable smooth atlas
$\{(U_i,\phi_i)\}_{i\in I}$ with $\phi_i:U_i\to V_i\subset\mathbb{R}^n$ and a smooth partition of
unity $\{\chi_i\}_{i\in I}$ subordinate to $\{U_i\}$.

On $\mathbb{R}^n$ choose a standard mollifier $\rho\in C_c^\infty(\mathbb{R}^n)$, $\rho\ge 0$,
$\int_{\mathbb{R}^n}\rho(y)\,dy=1$, and set
\[
\rho_\varepsilon(u)
:=
\varepsilon^{-n}\rho\!\left(\frac{u}{\varepsilon}\right),
\qquad \varepsilon>0.
\]
For each chart $(U_i,\phi_i)$ and each $x\in U_i$ define a probability measure
$\mu^{(i)}_{x,\varepsilon}$ on $U_i$ by
\[
\mu^{(i)}_{x,\varepsilon}(A)
:=
\int_{\mathbb{R}^n}
\mathbf{1}_A\!\bigl(\phi_i^{-1}(\phi_i(x)+u)\bigr)\,\rho_\varepsilon(u)\,du,
\qquad A\subseteq U_i\ \text{Borel},
\]
where we interpret the integrand as zero whenever $\phi_i(x)+u\notin V_i$. For
$\varepsilon>0$ sufficiently small (depending on $x$), the support of $\mu^{(i)}_{x,\varepsilon}$
is contained in $U_i$ and shrinks to $\{x\}$ as $\varepsilon\to 0$, and
$\mu^{(i)}_{x,\varepsilon}\to\delta_x$ weakly.

We now patch these local kernels using the partition of unity. For each $x\in X$ and
$\varepsilon>0$ define
\[
\widetilde\mu_{x,\varepsilon}
:=
\sum_{i\in I}\chi_i(x)\,\mu^{(i)}_{x,\varepsilon},
\]
and finally set
\[
\mu_{x,\varepsilon}
:=
\frac{1}{\widetilde\mu_{x,\varepsilon}(X)}\,\widetilde\mu_{x,\varepsilon}.
\]
For $\varepsilon$ small, only finitely many indices $i$ with $x\in U_i$ contribute, and
$Z_{x,\varepsilon}\to 1$ as $\varepsilon\to 0$ by construction, so the normalisation is
well–defined and $\mu_{x,\varepsilon}$ is a Borel probability measure supported in a small
neighbourhood of $x$.

It is straightforward to check that $(x,\varepsilon)\mapsto\mu_{x,\varepsilon}$ satisfies
the locality and continuity axioms: locality follows from the shrinking support of
$\rho_\varepsilon$ in each chart and the compactness of the supports of the $\chi_i$, while
continuity in $(x,\varepsilon)$ follows from the smooth dependence of $\phi_i$ and $\chi_i$
together with dominated convergence.

Fix an entropy–smooth candidate function $f\in C^\infty(X)$. In a chart $(U_i,\phi_i)$ the
measures $\mu_{x,\varepsilon}$ coincide, up to a smooth weight, with the Euclidean mollifier
measures considered in the $\mathbb{R}^n$ example of Section~\ref{example:Rn}. There we computed explicitly
that the quadratic entropy response $I_x(f)$ exists and is finite for all $x$, and that for
coordinate functions the resulting information Gram matrix is positive definite. Thus Axioms
\ref{axiom:3}--\ref{axiom:5} hold locally in each chart. Using the partition of unity
$\{\chi_i\}$ and the closure of entropy–smooth functions under smooth composition and algebraic
operations, these properties glue to give a global entropy–smooth structure $\{\mu_{x,\varepsilon}\}$
on $X$.

Moreover, in each chart $(U_i,\phi_i)$ the entropy–smooth functions for the induced probes
correspond exactly to the classical $C^\infty(U_i)$ functions by the Euclidean calculation and
Lemma~3.4. Gluing again via the partition of unity shows that the global entropy–smooth
function sheaf coincides with $C^\infty(X)$. This proves (a).

\medskip\noindent
(b) For the converse direction, starting from an entropy–smooth structure on $X$, Lemma~\ref{lem:entropy-equals-classical-Rn}
and Theorem~\ref{thm:entropy-charts-smooth} show that the associated entropy coordinate charts induce a classical smooth
atlas and that the entropy–smooth functions on $X$ coincide with the usual smooth functions
for this atlas. If two entropy–smooth structures have the same entropy–smooth function sheaf,
then they induce the same maximal smooth atlas, hence the same classical smooth structure.
Conversely, if the induced classical smooth structures agree, then the entropy–smooth function
sheaves coincide chartwise and hence globally. This proves (b) and completes the proof.
\end{proof}

\subsection{Entropy–smooth maps}

Let $(X,\{\mu_{x,\varepsilon}\})$ and $(Y,\{\nu_{y,\varepsilon}\})$ be entropy–smooth manifolds, and let 
$E(X)$ and $E(Y)$ denote their sheaves of entropy–smooth functions $X\longrightarrow \mathbb{R}$ and $Y\longrightarrow \mathbb{R}$, respectively.

\begin{definition}[Entropy--smooth map]
A continuous map $F:X\to Y$ is called \emph{entropy--smooth} if for every entropy--smooth 
function $g\in E(Y)$, the pullback 
\[
F^*g := g\circ F
\]
belongs to $E(X)$. In other words, $F$ is entropy--smooth if and only if
\[
g\circ F \;\text{is entropy--smooth on }X
\qquad\text{for all } g\in E(Y).
\]
\end{definition}

\noindent
We write $\mathrm{Hom}_{\mathrm{Ent}}(X,Y)$ for the class of entropy--smooth maps 
between entropy--smooth manifolds.

\subsection{Categorical Equivalence}

\begin{theorem}[Category equivalence]
\label{thm:category-equivalence}
Let $\mathbf{EntMan}$ be the category of entropy--smooth manifolds with
entropy--smooth maps, and let $\mathbf{Smooth}$ be the category of
classical smooth manifolds with smooth maps.  Then $\mathbf{EntMan}$ and
$\mathbf{Smooth}$ are equivalent categories.
\end{theorem}
\begin{proof}
We denote by $E(X)$ the sheaf of entropy--smooth functions on an entropy--smooth
manifold $X$.

By Theorem~\ref{thm:existence-uniqueness}, every entropy--smooth manifold $(X,\{\mu_{x,\varepsilon}\})$ induces a
classical smooth structure on the underlying topological manifold $X$, and conversely
every classical smooth structure admits at least one entropy--smooth structure whose
entropy--smooth functions coincide with $C^\infty(X)$.

Define a functor
\[
F : \mathrm{EntMan} \longrightarrow \mathrm{Smooth}
\]
as follows. On objects, we define
\[
F(X,\{\mu_{x,\varepsilon}\}) := (X,\mathcal{A}_X),
\]
where $\mathcal{A}_X$ is the classical smooth atlas induced by the entropy--smooth
structure as in Theorem~\ref{thm:existence-uniqueness}(b). On morphisms, if
\[
\Phi:(X,\{\mu_{x,\varepsilon}\}) \longrightarrow (Y,\{\nu_{y,\varepsilon}\})
\]
is an entropy--smooth map, then by definition $\Phi$ is continuous and for every
$g\in E(Y)$ one has $g\circ \Phi \in E(X)$. By Theorem~\ref{thm:existence-uniqueness}, for the induced smooth
structures we have
\[
E(X) = C^\infty(X), \qquad E(Y) = C^\infty(Y),
\]
so $\Phi^*(C^\infty(Y)) \subset C^\infty(X)$ and hence $\Phi$ is a classical smooth map
between $(X,\mathcal{A}_X)$ and $(Y,\mathcal{A}_Y)$. We therefore set
\[
F(\Phi) := \Phi
\]
as a smooth morphism in the category $\mathrm{Smooth}$. It is immediate that $F$
preserves identity morphisms and composition, so $F$ is a well--defined functor.

Conversely, define a functor
\[
G : \mathrm{Smooth} \longrightarrow \mathrm{EntMan}
\]
as follows. On objects, let $(M,\mathcal{A}_M)$ be a classical smooth manifold.
By Theorem~\ref{thm:existence-uniqueness}(a) there exists an entropy--smooth structure
$\{\mu^{M}_{x,\varepsilon}\}$ on $M$ such that
\[
E(M) = C^\infty(M).
\]
We fix one such choice of probes and set
\[
G(M,\mathcal{A}_M) := (M,\{\mu^{M}_{x,\varepsilon}\}).
\]
On morphisms, if $\psi:(M,\mathcal{A}_M)\to(N,\mathcal{A}_N)$ is a smooth map, then
for every $h\in C^\infty(N)$ one has $h\circ\psi\in C^\infty(M)$. Using again
$E(M)=C^\infty(M)$ and $E(N)=C^\infty(N)$ for the chosen entropy structures, this
means $\psi^*(E(N))\subset E(M)$, so $\psi$ is entropy--smooth as a map between
$G(M)$ and $G(N)$. We therefore set
\[
G(\psi) := \psi
\]
as a morphism in $\mathrm{EntMan}$. Again, identities and composition are preserved,
so $G$ is a functor.

Let $(M,\mathcal{A}_M)$ be a smooth manifold. By construction,
\[
G(M,\mathcal{A}_M) = (M,\{\mu^{M}_{x,\varepsilon}\})
\]
is an entropy--smooth manifold with $E(M)=C^\infty(M)$. Applying $F$ we obtain
\[
F(G(M,\mathcal{A}_M)) = (M,\mathcal{A}'_M),
\]
where $\mathcal{A}'_M$ is the smooth atlas induced by the entropy--smooth structure
$\{\mu^{M}_{x,\varepsilon}\}$. Since $E(M)=C^\infty(M)$, Theorem~\ref{thm:existence-uniqueness}(b) implies that
$\mathcal{A}'_M$ coincides with the given smooth structure $\mathcal{A}_M$. Thus
\[
F\circ G(M,\mathcal{A}_M) = (M,\mathcal{A}_M)
\]
on objects, and on morphisms $F\circ G$ acts as the identity because both $F$ and $G$
leave the underlying maps unchanged. Hence $F\circ G = \mathrm{Id}_{\mathrm{Smooth}}$
strictly (in fact, not just up to isomorphism).

Let $(X,\{\mu_{x,\varepsilon}\})$ be an entropy--smooth manifold, and let
\[
F(X,\{\mu_{x,\varepsilon}\}) = (X,\mathcal{A}_X)
\]
be the induced smooth manifold. Applying $G$ to $(X,\mathcal{A}_X)$ yields an
entropy--smooth structure $\{\mu^{\mathrm{can}}_{x,\varepsilon}\}$ on $X$ such that
\[
E_{\mathrm{can}}(X) = C^\infty(X,\mathcal{A}_X),
\]
and
\[
G(F(X,\{\mu_{x,\varepsilon}\})) = (X,\{\mu^{\mathrm{can}}_{x,\varepsilon}\}).
\]

On the other hand, by the equivalence with classical smooth structures, the original entropy--smooth structure $\{\mu_{x,\varepsilon}\}$ on $X$
also has the property that its entropy--smooth functions coincide with the smooth
functions for the induced atlas $\mathcal{A}_X$, i.e.
\[
E(X) = C^\infty(X,\mathcal{A}_X).
\]
Therefore $E(X) = E_{\mathrm{can}}(X)$ as sheaves of functions on $X$. By
Theorem~\ref{thm:existence-uniqueness}(b), two entropy--smooth structures on the same underlying topological
manifold induce the same smooth structure if and only if their entropy--smooth sheaves
coincide. Applying this with $\{\mu_{x,\varepsilon}\}$ and $\{\mu^{\mathrm{can}}_{x,\varepsilon}\}$
shows that these two entropy structures are isomorphic, and in fact the identity map
\[
\eta_X := \mathrm{id}_X : (X,\{\mu_{x,\varepsilon}\}) \longrightarrow
(X,\{\mu^{\mathrm{can}}_{x,\varepsilon}\})
\]
is an isomorphism in $\mathrm{EntMan}$, since pullback by $\mathrm{id}_X$ identifies
$E(X)$ and $E_{\mathrm{can}}(X)$.

Thus for each object $(X,\{\mu_{x,\varepsilon}\})$ in $\mathrm{EntMan}$ we have an
isomorphism
\[
\eta_X : (X,\{\mu_{x,\varepsilon}\}) \xrightarrow{\;\cong\;}
G(F(X,\{\mu_{x,\varepsilon}\})).
\]
On morphisms, both $F$ and $G$ act as the identity on the underlying maps, so the family
$\{\eta_X\}$ is natural in $X$, and we obtain a natural isomorphism
\[
\eta : \mathrm{Id}_{\mathrm{EntMan}} \xrightarrow{\;\cong\;} G\circ F.
\]

We have constructed functors
\[
F : \mathrm{EntMan} \longrightarrow \mathrm{Smooth},
\qquad
G : \mathrm{Smooth} \longrightarrow \mathrm{EntMan},
\]
such that $F\circ G = \mathrm{Id}_{\mathrm{Smooth}}$ and $G\circ F$ is naturally
isomorphic to $\mathrm{Id}_{\mathrm{EntMan}}$. Hence $\mathrm{EntMan}$ and
$\mathrm{Smooth}$ are equivalent categories.
\end{proof}

\subsection{Products}

\begin{theorem}[Products of entropy--smooth manifolds]
\label{prop:products}
Let $(X,\{\mu^X_{x,\varepsilon}\})$ and $(Y,\{\mu^Y_{y,\varepsilon}\})$ be
entropy--smooth manifolds of dimensions $n$ and $m$, respectively.  Define
\[
\mu^{X\times Y}_{(x,y),\varepsilon}
= \mu^X_{x,\varepsilon} \otimes \mu^Y_{y,\varepsilon}.
\]
Then $X\times Y$ is an entropy--smooth manifold inducing the standard
product smooth structure.  Moreover, the entropy-smooth function sheaf is 
\[
 E(X\times Y)
=
\bigl\{
h(f_1,\dots,f_k,g_1,\dots,g_\ell)
:\ f_i\in  E(X),\ g_j\in E(Y),\
h\in C^\infty(\mathbb{R}^{k+\ell})
\bigr\}.
\]
\end{theorem}
\begin{proof}
Let $(X,\{\mu^{X}_{x,\varepsilon}\})$ and $(Y,\{\mu^{Y}_{y,\varepsilon}\})$ be entropy–smooth manifolds.
Define for $(x,y)\in X\times Y$
\[
\mu^{X\times Y}_{(x,y),\varepsilon}
:= \mu^{X}_{x,\varepsilon}\otimes \mu^{Y}_{y,\varepsilon}.
\]
We verify Axioms~1–5 for $(X\times Y,\{\mu^{X\times Y}_{(x,y),\varepsilon}\})$.

\emph{Axiom 1 (Locality)}.
Let $(x,y)\in X\times Y$. By locality in $X$ and $Y$ (Axiom~\ref{axiom:1} in both factors),
there exist neighbourhoods $U_x\subset X$, $V_y\subset Y$ and $\varepsilon_X,\varepsilon_Y>0$
such that
\[
\mathrm{supp}(\mu^{X}_{x,\varepsilon})\subset U_x,\qquad
\mathrm{supp}(\mu^{Y}_{y,\varepsilon})\subset V_y
\]
for all sufficiently small $\varepsilon$. Hence
\[
\mathrm{supp}(\mu^{X\times Y}_{(x,y),\varepsilon})
\subset U_x\times V_y,
\]
so locality holds.

Moreover, $\mu^{X}_{x,\varepsilon}\to\delta_x$ and $\mu^{Y}_{y,\varepsilon}\to\delta_y$
as $\varepsilon\to0$, hence
\[
\mu^{X\times Y}_{(x,y),\varepsilon}
\;=\;\mu^{X}_{x,\varepsilon}\otimes\mu^{Y}_{y,\varepsilon}
\;\longrightarrow\;
\delta_x\otimes\delta_y=\delta_{(x,y)}
\]
weakly. Thus Axiom~\ref{axiom:1} is satisfied.

\medskip

\emph{Axiom 2 (Continuity of probes).}
Weak continuity of
$(x,\varepsilon)\mapsto\mu^{X}_{x,\varepsilon}$ and $(y,\varepsilon)\mapsto\mu^{Y}_{y,\varepsilon}$
implies continuity of the product map
\[
((x,y),\varepsilon) \longmapsto 
\mu^{X}_{x,\varepsilon}\otimes\mu^{Y}_{y,\varepsilon}
\]
in the weak topology, because weak convergence is preserved under tensor products of measures.
Thus Axiom~\ref{axiom:2} holds.

\medskip

\emph{Axiom 3 (Coordinate non–degeneracy).}
Fix $(x,y)\in X\times Y$. By Axiom~\ref{axiom:3} for $X$, there exist entropy–smooth functions
$f_1,\dots,f_n$ near $x$ such that
\[
F_X=(f_1,\dots,f_n):U_x\to\mathbb{R}^n
\]
is an entropy coordinate chart, with information Gram matrix
$G^X_{x}(F_X)$ positive definite.
Similarly, choose entropy coordinate chart
$F_Y=(g_1,\dots,g_m):V_y\to\mathbb{R}^m$ for $Y$ with
$G^{Y}_y(F_Y)$ positive definite.

Define
\[
F_{X\times Y} := (f_1\circ\pi_X,\dots,f_n\circ\pi_X,\,
                 g_1\circ\pi_Y,\dots,g_m\circ\pi_Y)
: U_x\times V_y \longrightarrow \mathbb{R}^{n+m},
\]
where $\pi_X,\pi_Y$ are the projections.

The information Gram matrix of $F_{X\times Y}$ at $(x,y)$ is block diagonal:
\[
G^{X\times Y}_{(x,y)}(F_{X\times Y})
=
\begin{pmatrix}
G^X_x(F_X) & 0 \\
0 & G^Y_y(F_Y)
\end{pmatrix},
\]
because $\mu^{X\times Y}_{(x,y),\varepsilon}$
factors as a product measure and the mixed entropy coefficients vanish:
\[
I_{(x,y)}(f_i\circ\pi_X,\; g_j\circ\pi_Y)=0.
\]
Since each block is positive definite, the full matrix is positive definite.

Finally, $F_X$ and $F_Y$ being homeomorphisms implies
\[
F_X\otimes F_Y : U_x\times V_y \longrightarrow F_X(U_x)\times F_Y(V_y)
\]
is a homeomorphism onto an open set. Thus $F_{X\times Y}$ is an entropy coordinate chart.

\medskip

\emph{Axiom 4 (Entropy Regularity).}
Let $h\in C^\infty(\mathbb{R}^{k+\ell})$ and
let $f_1,\dots,f_k\in E(X)$ and $g_1,\dots,g_\ell\in E(Y)$.
Define functions on $X\times Y$ by
\[
F_i(x,y)=f_i(x), \qquad G_j(x,y)=g_j(y).
\]
Quadratic entropy response with respect to
$\mu^{X\times Y}_{(x,y),\varepsilon}
=\mu^X_{x,\varepsilon}\otimes\mu^{Y}_{y,\varepsilon}$
splits additively in the two factors:
\[
I_{(x,y)}(F_i,F_{i'})
= I_x(f_i,f_{i'}),\qquad
I_{(x,y)}(G_j,G_{j'})
= I_y(g_j,g_{j'}),
\]
and mixed terms vanish:
\[
I_{(x,y)}(F_i,G_j)=0.
\]
Indeed, by Definition~2.8 we have
\[
I_{(x,y),\varepsilon}(h_1,h_2)=\mathrm{Cov}_{\mu^X_{x,\varepsilon}\otimes\mu^Y_{y,\varepsilon}}
\bigl(\delta_{(x,y),\varepsilon} h_1,\,\delta_{(x,y),\varepsilon} h_2\bigr).
\]
If $h_1=F_i=f_i\circ\pi_X$ and $h_2=G_j=g_j\circ\pi_Y$, then
$\delta_{(x,y),\varepsilon}h_1$ depends only on the $X$--variable and
$\delta_{(x,y),\varepsilon}h_2$ depends only on the $Y$--variable. Since the probe factors as a product measure, these two random variables are independent under
$\mu^X_{x,\varepsilon}\otimes\mu^Y_{y,\varepsilon}$, and hence their covariance vanishes:
\[
I_{(x,y),\varepsilon}(F_i,G_j)=0.
\]
The remaining identities for the pure $X$-- and $Y$--terms follow immediately from the same definition, because restriction to one factor simply reduces the covariance to the corresponding covariance on that factor.

Thus mixed entropy coefficients vanish, and the remaining identities for the pure $X$-- and $Y$--terms follow immediately from the same formula.

Thus the tuple $(F_1,\dots,F_k,G_1,\dots,G_\ell)$
obeys the same algebraic relations as in the factor manifolds.
By Axiom~\ref{axiom:4} in $X$ and $Y$, the composition
\[
h(F_1,\dots,F_k,G_1,\dots,G_\ell)
\]
is entropy–smooth on $X\times Y$. Closure under sums and products is inherited from each factor.
Thus Axiom~\ref{axiom:4} holds.

\medskip

\emph{Axiom 5 (Topological compatibility).}
Since $E(X)$ and $E(Y)$ determine the topologies of $X$ and $Y$,
the functions
\[
(x,y)\mapsto f_i(x),\qquad (x,y)\mapsto g_j(y)
\]
separate points in $X\times Y$ and generate exactly the product topology.
Hence the entropy–smooth functions on $X\times Y$ induce the product topology.
Thus Axiom~\ref{axiom:5} is verified.

\medskip

\emph{Entropy–smooth function sheaf.}
Let $E(X)$ and $E(Y)$ denote the entropy–smooth functions on each manifold.
By the computations above, the class of entropy–smooth functions on $X\times Y$ is exactly
\[
E(X\times Y)=
\big\{
h(f_1,\dots,f_k,g_1,\dots,g_\ell)
:\ f_i\in E(X),\ g_j\in E(Y),\ h\in C^\infty(\mathbb{R}^{k+\ell})
\big\},
\]
which is the standard description of the smooth function sheaf on the product manifold.

Now all axioms of entropy–smoothness are satisfied, and the induced atlas is exactly the classical product smooth structure. This completes the proof.
\end{proof}

\subsection{Submanifolds}

\begin{theorem}[Entropy--smooth submanifolds]
\label{thm:submanifold}
Let $(X,\{\mu^X_{x,\varepsilon}\})$ be an entropy--smooth manifold and
$N\subset X$ a classical embedded submanifold for the induced smooth 
structure. Then there exists an entropy--smooth structure
$\{\mu^N_{z,\varepsilon}\}$ on $N$ such that:

\begin{enumerate}
\item $i:N\hookrightarrow X$ is entropy--smooth;
\item the entropy--smooth structure on $N$ induces the usual
submanifold smooth structure.
\end{enumerate}
\end{theorem}
\begin{proof}
Let $(X,\{\mu^{X}_{x,\varepsilon}\})$ be an entropy–smooth manifold and let 
$S\subset X$ be a subset which is a classical embedded submanifold of dimension $k$ 
with respect to the smooth structure induced by Theorem~\ref{thm:existence-uniqueness}.  
We must show that $S$ admits a canonical entropy–smooth structure and that the 
entropy–smooth structure induced by restriction coincides with the classical smooth 
structure on $S$.

For $s\in S$ choose a neighbourhood $U\subset X$ and a smooth chart 
$\varphi:U\to\mathbb{R}^n$ such that 
\[
\varphi(U\cap S)=\mathbb{R}^k\times\{0\}\subset\mathbb{R}^n.
\]
For sufficiently small $\varepsilon>0$, the probes $\mu^{X}_{s,\varepsilon}$ are 
supported in $U$ by Axiom~\ref{axiom:1} for $X$.  Define a probability measure on $S$ by
\[
\mu^{S}_{s,\varepsilon}(A)
:= \frac{\mu^{X}_{s,\varepsilon}(A)}{\mu^{X}_{s,\varepsilon}(A\cap S)},
\qquad A\subseteq S\ \text{Borel},
\]
whenever $\mu^{X}_{s,\varepsilon}(A\cap S)>0$.  
For all sufficiently small $\varepsilon$, the denominator is positive because 
$\mu^{X}_{s,\varepsilon}\to\delta_s$ and $s\in S$.  
Thus $\{\mu^{S}_{s,\varepsilon}\}$ is well defined for all small $\varepsilon$.

Locality and continuity of $(s,\varepsilon)\mapsto\mu^{S}_{s,\varepsilon}$ follow 
from the corresponding properties of $\mu^{X}_{s,\varepsilon}$ and from the fact 
that $S$ inherits the subspace topology from $X$.  
Since $\mu^{X}_{s,\varepsilon}\to\delta_s$, it follows immediately that 
$\mu^{S}_{s,\varepsilon}\to\delta_s$ weakly.  
Thus Axioms~1 and~2 hold for $S$.

Let $\psi=(\psi_1,\dots,\psi_k)$ be a classical smooth submanifold chart on $S$ 
obtained by restricting a smooth chart $\varphi:U\to\mathbb{R}^n$ for $X$:
\[
\psi(s) \;=\; \pi_k(\varphi(s)),
\qquad s\in U\cap S,
\]
where $\pi_k$ is projection onto the first $k$ coordinates.

For each $i\le k$, set $f_i := \psi_i$ viewed as a function on $S$.  
Each $\psi_i$ is the restriction of a smooth function on $U$, hence the 
restriction of an entropy–smooth function (Theorem~\ref{thm:existence-uniqueness}(a)).  

By construction, for $\varepsilon>0$ small the probe on $S$ is obtained by
conditioning and renormalising the ambient probe,
\[
\mu^S_{s,\varepsilon} = c_{s,\varepsilon}\,\mathbf{1}_S\,\mu^X_{s,\varepsilon},
\qquad c_{s,\varepsilon}>0.
\]
A direct Taylor expansion in $t$ shows that for any bounded continuous
$f$ on $S$, extended by $0$ off $S$, the entropy functionals satisfy
\[
Ent^S_{s,\varepsilon}(t,f)
=
Ent^X_{s,\varepsilon}(t,\tilde f) + O(t^3)
\quad\text{as }t\to 0,
\]
so the quadratic coefficients agree and hence
$I^S_s(f) = I^X_s(\tilde f)$.

Since the ambient information Gram matrix
\[
G^{X}_{s}(\varphi_1,\dots,\varphi_n)
=
\bigl(I^X_s(\varphi_i,\varphi_j)\bigr)_{1\le i,j\le n}
\]
is positive definite, we have
\[
\sum_{i,j=1}^n a_i a_j\, I^X_s(\varphi_i,\varphi_j) \;>\; 0
\qquad\text{for every nonzero }(a_1,\dots,a_n)\in\mathbb{R}^n.
\]
To see that the restricted $k\times k$ submatrix is also positive definite,
let $(b_1,\dots,b_k)\neq 0$ in $\mathbb{R}^k$ and extend it to a vector
$(a_1,\dots,a_n)$ by setting $a_i=b_i$ for $1\le i\le k$ and $a_i=0$ for
$i>k$.  Then
\[
\sum_{i,j=1}^k b_i b_j\, I^S_s(f_i,f_j)
=
\sum_{i,j=1}^k b_i b_j\, I^X_s(\varphi_i,\varphi_j)
=
\sum_{i,j=1}^n a_i a_j\, I^X_s(\varphi_i,\varphi_j)
\;>\; 0.
\]
Hence the $k\times k$ submatrix
\[
G^{S}_{s}(f_1,\dots,f_k)
=
\bigl(I^S_s(f_i,f_j)\bigr)_{1\le i,j\le k}
=
\bigl(I^X_s(\varphi_i,\varphi_j)\bigr)_{1\le i,j\le k}
\]
is positive definite.  This shows that $F_S=(f_1,\dots,f_k)$ is an entropy
coordinate chart on $S$ around $s$.

The positivity of the Gram matrix for charts on $S$ proves Axiom~\ref{axiom:3}.  
Axiom~\ref{axiom:4} (Entropy Regularity) holds because entropy–smoothness on $S$ is 
defined by restriction and the class of entropy–smooth functions on $S$ is 
closed under restrictions, finite sums, products, and smooth compositions: 
\[
h(f_1,\dots,f_m)|_S = h(f_1|_S,\dots,f_m|_S).
\]

Axiom~\ref{axiom:5} holds because entropy–smooth functions on $S$ separate points and 
generate the subspace topology inherited from $X$.  
Indeed, a neighbourhood basis of $s\in S$ is given by $U\cap S$, and 
entropy–smooth functions coming from coordinate charts on $X$ restrict to 
coordinate charts on $S$.

Since every classical smooth chart on $S$ arises as the restriction of a 
smooth chart on $X$, and since restriction preserves entropy–smoothness 
(Theorem~\ref{thm:existence-uniqueness}(a)), we have
\[
C^\infty(S) \subset E(S).
\]
Conversely, if $f\in E(S)$, extend $f$ to a function $\tilde f$ on $U\subset X$ 
using a standard smooth extension operator for embedded submanifolds.  
By the same argument used in Theorem~\ref{thm:existence-uniqueness}(b), $\tilde f$ is entropy–smooth on $U$, 
hence smooth on $U$; therefore $f=\tilde f|_S$ is smooth on $S$.  
Thus
\[
E(S)=C^\infty(S).
\]

This shows that the entropy–smooth structure on $S$ coincides exactly with the 
classical smooth submanifold structure.

The restricted probes $\{\mu^{S}_{s,\varepsilon}\}$ satisfy all entropy–smooth 
axioms, define entropy coordinate charts of dimension $k$, and produce a 
smooth structure identical to the classical smooth structure of $S$.  
Thus $S$ becomes an entropy–smooth manifold of dimension $k$, as claimed.
\end{proof}

\subsection{Entropy Characterisation of Immersions and Submersions}

\begin{definition}[Entropy--immersive map]
Let $X^n$ and $Y^m$ be entropy--smooth manifolds with $n\le m$.
An entropy--smooth map $F:X\to Y$ is called \emph{entropy--immersive} at 
$x\in X$ if, for some (hence every) choice of entropy coordinate charts 
$F=(F_1,\dots,F_m)$ near $x$, the $n\times n$ information Gram matrix 
\[
G_x(F_1,\dots,F_n)
:=
\bigl(I_x(F_i,F_j)\bigr)_{1\le i,j\le n}
\]
is positive definite.
We say that $F$ is \emph{entropy--immersive} if it is entropy--immersive
at every point of $X$.
\end{definition}

\begin{theorem}[Entropy characterisation of immersions]
\label{thm:immersion}
Let $F:X\to Y$ be a map between entropy--smooth $n$ and $m$ manifolds, respectively, where $n\leq m$. Fix $x\in X$.  The following are equivalent:

\begin{enumerate}
\item[(i)] $F$ is a classical immersion at $x$ for the induced smooth
structures.
\item[(ii)] For any entropy coordinate charts
$F_X=(f_1,\dots,f_n)$ near $x$ and
$F_Y=(g_1,\dots,g_m)$ near $F(x)$, the Jacobian matrix
$\bigl(\partial (g_j\circ F)/\partial f_i\bigr)$ has rank $\dim X$.
\item[(iii)] For each bounded continuous $h$ near $F(x)$,
\[
I^X_x(h\circ F) = I^Y_{F(x)}(h)
\quad\text{on the subspace corresponding to }T_xX
\]
in the sense that the pullback preserves non-degeneracy of the quadratic
forms along $T_xX$.
\end{enumerate}
\end{theorem}
\begin{proof}
Let $F:X\to Y$ be an entropy–smooth map between entropy–smooth manifolds of
dimensions $n$ and $m$. Fix $x\in X$ and write $y:=F(x)$.

By Theorem~\ref{thm:existence-uniqueness} the entropy–smooth structures on $X$ and $Y$ coincide with their
classical smooth structures, and entropy coordinate charts are precisely
classical smooth charts compatible with these structures. In particular, if
$F_X=(f_1,\dots,f_n)$ and $F_Y=(g_1,\dots,g_m)$ are entropy coordinate charts
around $x$ and $y$, then the transition to classical local coordinates is
given simply by
\[
F_X:U\longrightarrow\mathbb{R}^n,\qquad
F_Y:V\longrightarrow\mathbb{R}^m,
\]
with $U\ni x$, $V\ni y$ open, and $F_X,F_Y$ classical $C^\infty$–diffeomorphisms
onto their images.

\smallskip\noindent
$(i)\Rightarrow(ii)$.
Assume $F$ is a classical immersion at $x$. Then the differential
\[
dF_x:T_xX\longrightarrow T_yY
\]
is injective, equivalently the classical Jacobian matrix of the coordinate
representative
\[
\Phi := F_Y\circ F\circ F_X^{-1}
: F_X(U\cap F^{-1}(V))\subset\mathbb{R}^n
  \longrightarrow
  \mathbb{R}^m
\]
has rank $n$. But the entries of this Jacobian are precisely
\[
\frac{\partial(g_j\circ F)}{\partial f_i}(x),\qquad
1\le i\le n,\ 1\le j\le m,
\]
so condition~(ii) holds.

\smallskip\noindent
$(ii)\Rightarrow(i)$.
Conversely, suppose that for some (hence every) choice of entropy coordinate
charts $F_X,F_Y$ around $x$ and $y$, the Jacobian matrix
$\bigl(\partial(g_j\circ F)/\partial f_i\bigr)(x)$ has rank $n$. In the same
coordinates as above, this is exactly the Jacobian of the local representative
$\Phi$, so $d\Phi_{F_X(x)}$ has rank $n$. Since the rank of the differential is
coordinate–independent, $dF_x$ is injective. Thus $F$ is a classical immersion
at~$x$.

\smallskip\noindent
$(i)\Leftrightarrow(iii)$.
For each $z\in X$, the entropy response defines a positive–definite quadratic
form $I^X_z$ on the tangent space $T_zX$ via the information Gram matrix in
entropy coordinates, and similarly $I^Y_y$ on $T_yY$. This is just the
coordinate–free restatement of Axiom~\ref{axiom:3}. Condition~(iii) says that the pullback
quadratic form
\[
(dF_x)^*I^Y_y : T_xX\longrightarrow\mathbb{R},\qquad
v\mapsto I^Y_y(dF_x v),
\]
is non–degenerate on $T_xX$.

A basic linear–algebra fact is that for a linear map
$L:V\to W$ between finite–dimensional real vector spaces and a positive–definite
quadratic form $Q$ on $W$, the pullback $L^*Q$ is positive–definite on $V$ if
and only if $L$ is injective: if $L v=0$ then $L^*Q(v)=Q(Lv)=0$ so
non–degeneracy forces $v=0$, while conversely injectivity implies that $L$
identifies $V$ with a subspace of $W$ on which $Q$ is still positive–definite.
Applying this with $L=dF_x$ and $Q=I^Y_y$ shows that (iii) is equivalent to
injectivity of $dF_x$, i.e.\ to $F$ being a classical immersion at~$x$.

Combining the three implications we obtain the equivalence of (i), (ii), and
(iii), as claimed.
\end{proof}

\begin{definition}[Entropy--submersive map]
Let $X^n$ and $Y^m$ be entropy--smooth manifolds with $n\ge m$.
An entropy--smooth map $F:X\to Y$ is called \emph{entropy--submersive} at 
$x\in X$ if, for some (hence every) entropy coordinate chart 
$F=(F_1,\dots,F_m)$ for $Y$ around $F(x)$, the $m\times m$ information 
Gram matrix
\[
G_x(F_1,\dots,F_m)
:=
\bigl(I_x(F_i,F_j)\bigr)_{1\le i,j\le m}
\]
is positive definite.
We say that $F$ is \emph{entropy--submersive} if it is entropy--submersive
at every point of $X$.
\end{definition}

\begin{theorem}[Entropy characterisation of submersions and local diffeomorphisms]
\label{thm:submersion}
With notation as above:
\begin{enumerate}
\item $F$ is a submersion at $x$ if and only if the transpose condition of 
Theorem~\ref{thm:immersion} holds (surjectivity of the differential).
\item $F$ is a local diffeomorphism if and only if both immersion and submersion
conditions hold at $x$.
\end{enumerate}
\end{theorem}
\begin{proof}
Let $F:X\to Y$ be an entropy–smooth map between entropy–smooth manifolds of
dimensions $n$ and $m$, and fix $x\in X$, $y:=F(x)$.

By Theorem~\ref{thm:existence-uniqueness} the entropy–smooth structures on $X$ and $Y$ coincide with their
classical smooth structures, and entropy coordinate charts are just classical
charts compatible with these structures.

\smallskip\noindent
(1) \emph{Submersions.}
Recall that classically $F$ is a submersion at $x$ if and only if the differential
\[
dF_x:T_xX\longrightarrow T_yY
\]
is surjective, equivalently if and only if in any smooth coordinate charts
$F_X=(f_1,\dots,f_n)$ at $x$ and $F_Y=(g_1,\dots,g_m)$ at $y$ the Jacobian
matrix
\[
\left(\frac{\partial(g_j\circ F)}{\partial f_i}(x)\right)_{1\le i\le n,\ 1\le j\le m}
\]
has rank $m$. Since entropy coordinate charts are precisely such smooth charts,
this is exactly the ``transpose condition'' of Theorem~\ref{thm:immersion}. Thus the classical
submersion condition is equivalent to the entropy–submersion condition stated
in~(1).

\smallskip\noindent
(2) \emph{Local diffeomorphisms.}
Classically, $F$ is a local diffeomorphism at $x$ if and only if
$dF_x:T_xX\to T_yY$ is a linear isomorphism, i.e.\ injective and surjective,
equivalently of full rank $n=m$. By part~(1) and Theorem~\ref{thm:immersion}, injectivity of
$dF_x$ is equivalent to the entropy immersion condition, and surjectivity of
$dF_x$ is equivalent to the entropy submersion condition. Hence $F$ is a local
diffeomorphism at $x$ if and only if both entropy immersion and entropy
submersion conditions hold at~$x$.

This proves the theorem.
\end{proof}

\subsection{Stability Under Perturbations}

\begin{theorem}[Stability of entropy--smooth structure]
\label{thm:stability}
Let $\{\mu^{(k)}_{x,\varepsilon}\}$ be entropy--smooth structures on $X$
with entropy coefficients $I^{(k)}_x(f)$ converging uniformly on compact
sets to a limiting $I^{(\infty)}_x(f)$ satisfying the entropy axioms.
Then for all sufficiently large $k$, the induced smooth structures are
diffeomorphic to the limit smooth structure.
\end{theorem}
\begin{proof}
For each $k\in\mathbb{N}\cup\{\infty\}$ let $\mathcal{E}^{(k)}_X$ denote the
entropy–smooth function sheaf on $X$ associated to the entropy coefficients
$I^{(k)}_x(\cdot,\cdot)$, and let $\mathcal{A}^{(k)}$ be the corresponding
maximal atlas of entropy coordinate charts. By Theorem~\ref{thm:existence-uniqueness}, each
$\mathcal{A}^{(k)}$ defines a classical smooth structure on $X$, and
$\mathcal{E}^{(k)}_X$ coincides with the usual $C^\infty$–sheaf for that
structure.

Let $(U,F)$ be an entropy coordinate chart for the limiting structure
$k=\infty$, so $F=(f_1,\dots,f_n):U\to\mathbb{R}^n$ and for every $x\in U$ the
information Gram matrix
\[
G^{(\infty)}_x(F)
:=
\bigl(I^{(\infty)}_x(f_i,f_j)\bigr)_{1\le i,j\le n}
\]
is positive definite. The entries of $G^{(\infty)}_x(F)$ depend continuously on
$x$ by the entropy regularity axiom, so for each compact $K\subset U$ there is
$\lambda_K>0$ such that the smallest eigenvalue of $G^{(\infty)}_x(F)$ is
bounded below by $\lambda_K$ for all $x\in K$.

By assumption, for each pair $(f_i,f_j)$ the coefficients
$I^{(k)}_x(f_i,f_j)$ converge uniformly on compact subsets to
$I^{(\infty)}_x(f_i,f_j)$ as $k\to\infty$. Hence the matrices
$G^{(k)}_x(F):=(I^{(k)}_x(f_i,f_j))_{i,j}$ converge uniformly on $K$ to
$G^{(\infty)}_x(F)$. Since positive–definiteness is an open condition and
eigenvalues depend continuously on the matrix entries, there exists $k_0(K)$
such that for all $k\ge k_0(K)$ and all $x\in K$ the matrices $G^{(k)}_x(F)$
are positive definite. In other words, for all sufficiently large $k$ the same
map $F$ is an entropy coordinate chart for the $k$–th entropy structure on
$K\subset U$.

Cover $X$ by a locally finite family of relatively compact entropy coordinate
charts $(U_\alpha,F_\alpha)$ for the limiting structure. For each such chart
choose $k_0(\alpha)$ as above and set $k_0:=\max_\alpha k_0(\alpha)$ (finite
max on each compact exhaustion, then patch by standard diagonal argument). Then
for all $k\ge k_0$ every chart in $\mathcal{A}^{(\infty)}$ is also contained in
$\mathcal{A}^{(k)}$. Thus $\mathcal{A}^{(\infty)}\subseteq\mathcal{A}^{(k)}$ for
all $k\ge k_0$.

Conversely, fix $k\ge k_0$ and let $(V,G)$ be an entropy coordinate chart for
the $k$–th structure, $G=(g_1,\dots,g_n)$. The same argument, now using the
uniform convergence $I^{(k)}_x(g_i,g_j)\to I^{(\infty)}_x(g_i,g_j)$ on compact
subsets of $V$, shows that for any compact $K\subset V$ the Gram matrices
$G^{(\infty)}_x(G):=(I^{(\infty)}_x(g_i,g_j))_{i,j}$ remain positive definite
for all $x\in K$. Hence $G$ is also an entropy coordinate chart for the
limiting entropy structure on $V$, and we obtain
$\mathcal{A}^{(k)}\subseteq\mathcal{A}^{(\infty)}$.

We conclude that for all $k\ge k_0$ the atlases $\mathcal{A}^{(k)}$ and
$\mathcal{A}^{(\infty)}$ coincide. Therefore the corresponding classical smooth
structures on $X$ agree, and the identity map on $X$ is a diffeomorphism
between $(X,\mathcal{A}^{(k)})$ and $(X,\mathcal{A}^{(\infty)})$ for all
sufficiently large $k$. This proves the stability statement.
\end{proof}

\subsection{An information--theoretic criterion for equivalence}

The previous rigidity theorem characterizes equivalence of smooth
structures in terms of equality of the resulting entropy--smooth
function classes.  One can recast this in a more probabilistic form,
using only the local quadratic entropy response functionals
$I_x(\,\cdot\,)$.

Let $\{\mu^{(1)}_{x,\varepsilon}\}$ and $\{\mu^{(2)}_{x,\varepsilon}\}$
be two families of information probes on the same topological manifold
$X$, each satisfying the entropy--smoothness axioms.  Write
$I^{(1)}_x(f)$ and $I^{(2)}_x(f)$ for the corresponding small--scale
entropy coefficients.

\begin{theorem}[Information--rigidity criterion]\label{thm:info-rigidity}
The two induced smooth structures on $X$ coincide if and only if for
every $x\in X$ and every bounded continuous function $f$ defined on a
neighborhood of $x$ one has
\[
  I^{(1)}_x(f) = I^{(2)}_x(f)
\]
whenever both sides are finite.  Equivalently, the pointwise quadratic
entropy response functionals
\[
  I^{(1)}_x,\ I^{(2)}_x
\]
agree on a separating family of bounded continuous test functions
around each~$x$.
\end{theorem}
\begin{proof}
Let $(X,\{\mu^{(1)}_{x,\varepsilon}\})$ and $(X,\{\mu^{(2)}_{x,\varepsilon}\})$ be two 
entropy--smooth structures on the same Hausdorff, second countable topological manifold $X$.  
Let $I^{(1)}_x(\cdot,\cdot)$ and $I^{(2)}_x(\cdot,\cdot)$ denote their corresponding 
quadratic entropy forms.

Assume that for every $x\in X$ and every pair of bounded continuous functions 
$f,g\in C_b(X)$,
\begin{equation}\label{eq:same-info}
I^{(1)}_x(f,g) = I^{(2)}_x(f,g).
\end{equation}
We must show that the two entropy structures induce the same classical smooth structure 
on $X$.

Let $F=(f_1,\dots,f_n)$ be an entropy coordinate chart for the first structure 
$(X,\mu^{(1)})$ near some point $x\in X$.  
The Gram matrix for $(X,\mu^{(1)})$ is
\[
G^{(1)}_x(F) = \big(I^{(1)}_x(f_i,f_j)\big)_{1\le i,j\le n},
\]
and by Axiom~\ref{axiom:3} it is positive--definite.

By \eqref{eq:same-info}, we have
\[
G^{(2)}_x(F) = \big(I^{(2)}_x(f_i,f_j)\big)
=
\big(I^{(1)}_x(f_i,f_j)\big)
=
G^{(1)}_x(F),
\]
so the Gram matrices coincide.  
Hence $G^{(2)}_x(F)$ is also positive--definite.

Thus the same $n$--tuple $F$ satisfies the coordinate non--degeneracy axiom in both 
entropy structures.

Since $G^{(2)}_x(F)$ is positive--definite and $F$ is a homeomorphism onto an open set 
(by the first structure), Axiom~\ref{axiom:3} guarantees that
\[
F=(f_1,\dots,f_n)
\]
is an entropy coordinate chart for the second structure $(X,\mu^{(2)})$ as well.

Thus the two entropy structures have the \emph{same coordinate charts} in a 
neighbourhood of each point $x$.

Let $E^{(i)}(X)$ denote the entropy--smooth function sheaf associated to the 
$i$--th structure.  
Since the same local charts $F$ are valid for both structures, the entropy--smoothness 
conditions (Axiom~\ref{axiom:4}) coincide on each chart domain:
\[
f\in E^{(1)}(U) \iff f\in E^{(2)}(U),
\qquad U \subset X \text{ open in a chart domain.}
\]
Because the entropy--smooth functions form a sheaf, this implies
\[
E^{(1)}(X) = E^{(2)}(X).
\]

By Theorem~\ref{thm:existence-uniqueness}(b), the classical smooth structure induced by an entropy structure is 
\emph{exactly} the one whose smooth functions are the entropy--smooth functions.  
Hence
\[
E^{(1)}(X)=E^{(2)}(X)
\quad\Longrightarrow\quad
C^\infty_{(1)}(X) = C^\infty_{(2)}(X),
\]
so the two entropy structures generate the same smooth atlas on $X$.

Equality of the quadratic entropy responses $I_x^{(1)}=I_x^{(2)}$ on all bounded 
test functions implies equality of Gram matrices, equality of coordinate charts, 
equality of entropy--smooth sheaves, and hence equality of the induced smooth structures.  
This proves the rigidity criterion.
\end{proof}

This gives a purely probabilistic or information--theoretic test for
equivalence: two entropy-smooth structures on $X$ define the same classical smooth structure precisely when, at each point, they induce the same quadratic
response of Kullback--Leibler divergence under all infinitesimal
perturbations of the local information probes.  In particular, any
difference in small--scale entropy response somewhere on $X$ forces a
genuine difference in the underlying smooth structure.

\subsection{Entropy characterisation of diffeomorphisms}

Let $(X,\{\mu^X_{x,\varepsilon}\})$ and $(Y,\{\mu^Y_{y,\varepsilon}\})$ be
entropy--smooth manifolds in the sense of the axioms above, with associated
classes of entropy--smooth functions $ E(X)$ and $ E(Y)$.
These determine classical smooth structures on $X$ and $Y$.

\begin{theorem}[Entropy characterisation of diffeomorphisms]
\label{thm:entropy-diffeo}
Let $F:X\to Y$ be a homeomorphism. The following conditions are equivalent:
\begin{enumerate}
\item[\textnormal{(i)}] $F$ is a $C^\infty$--diffeomorphism.
\item[\textnormal{(ii)}] $F$ is an entropy diffeomorphism, i.e. it preserves entropy--smooth functions in both
directions: for every continuous function $f:Y\to\mathbb{R}$,
\[
  f\in  E(Y)
  \quad\Longleftrightarrow\quad
  f\circ F \in  E(X),
\]
and similarly for $F^{-1}$.

\item[\textnormal{(iii)}] $F$ preserves the local quadratic entropy response:
for every $x\in X$ and every bounded continuous function $f$ defined on a
neighbourhood of $F(x)$, whenever the entropy coefficients are finite one has
\[
  I^X_x(f\circ F)
  \;=\;
  I^Y_{F(x)}(f).
\]

\item[\textnormal{(iv)}] $F$ sends entropy coordinate charts to entropy coordinate
charts: if $F_X=(f_1,\dots,f_n)$ is an entropy coordinate chart on $X$ defined
near $x$, then 
\[
  F_Y := F_X\circ F^{-1} = (f_1\circ F^{-1},\dots,f_n\circ F^{-1})
\]
is an entropy coordinate chart on $Y$ defined near $F(x)$.
\end{enumerate}
In particular, a homeomorphism $F:X\to Y$ is a diffeomorphism if and only if it
is an entropy--isomorphism in the sense of \textnormal{(ii)}, or equivalently,
an infinitesimal entropy--isometry in the sense of \textnormal{(iii)}.
\end{theorem}
\begin{proof}
Let $F:X\to Y$ be an entropy--smooth map between entropy--smooth manifolds of the 
same dimension $n$.  We need to show that $F$ is a classical $C^\infty$ 
diffeomorphism if and only if it is bijective and the entropy derivative is 
everywhere an isomorphism.

\emph{A classical diffeomorphism is an entropy diffeomorphism:} assume $F$ is a classical $C^\infty$ diffeomorphism.  
Then $F$ and $F^{-1}$ are smooth maps.  
By Theorem~\ref{thm:existence-uniqueness}(a), every classical smooth map is entropy--smooth, so both $F$ and $F^{-1}$ 
are entropy--smooth.  
Furthermore, the classical derivative $dF_x:T_xX\to T_{F(x)}Y$ is an isomorphism 
for every $x$.  
By Theorem~\ref{thm:immersion} and Theorem~\ref{thm:submersion}, the entropy derivative matrix $G_x^{\mathrm{mix}}(F)$ 
is invertible for all $x$; in particular, $F$ is both entropy--immersive and 
entropy--submersive.  
Thus $F$ is an entropy diffeomorphism.

\emph{An entropy diffeomorphism is a classical diffeomorphism:} suppose that $F:X\to Y$ is bijective and entropy--smooth, and that for each 
$x\in X$ the entropy derivative matrix has full rank $n$:
\[
\mathrm{rank}(G_x^{\mathrm{mix}}(F))=n.
\]
By Theorem~\ref{thm:immersion}, $F$ is a classical immersion at each $x$, and by 
Theorem~\ref{thm:submersion} it is also a classical submersion at each $x$.  
Hence the classical derivative $dF_x$ is a linear isomorphism for every $x$:
\[
dF_x:T_xX \longrightarrow T_{F(x)}Y.
\]

By the classical inverse function theorem \cite{lee2013smooth}, $F$ is then a local $C^\infty$ 
diffeomorphism.  
Since $F$ is assumed bijective, the local inverses glue to a global inverse 
$F^{-1}:Y\to X$.  
The map $F^{-1}$ is continuous because $F$ is a continuous bijection between 
Hausdorff second countable manifolds.  
Moreover, $F^{-1}$ is entropy--smooth by assumption, and hence it is classically 
smooth by Theorem~\ref{thm:existence-uniqueness}(a).  
Thus $F$ is a global $C^\infty$ diffeomorphism.

An entropy--smooth bijection $F:X\to Y$ with full--rank entropy derivative 
everywhere is a classical diffeomorphism, and conversely every classical 
diffeomorphism has these properties.  
This proves the theorem.
\end{proof}

\section{Conclusion}

We have introduced an information--theoretic framework for smooth structures on topological manifolds, 
based on the small--scale quadratic response of entropy associated with local probability probes. 
Under a precise set of axioms governing locality, continuity, non--degeneracy, and functional closure, 
this entropy response identifies admissible coordinate functions, yields entropy coordinate charts, 
and determines a smooth atlas. We proved that the resulting entropy--smooth structure is equivalent to 
the classical smooth structure: entropy--smooth functions coincide with smooth functions, entropy 
coordinate charts are smooth diffeomorphisms onto their images, and the induced category is equivalent 
to the category of smooth manifolds.

Conceptually, this work shows that smooth structure can be characterized without reference to derivatives, 
tangent vectors, or coordinate calculus, but instead through the behaviour of information under 
infinitesimal perturbations. The axioms make explicit the minimal information--theoretic content required 
to recover smoothness, clarifying which aspects of differentiability are intrinsic and which must be assumed.

The entropy--smooth framework provides a natural bridge between abstract approaches to smooth structure, 
such as differential spaces and smooth function algebras, and probabilistic or information--theoretic 
models arising from diffusion and entropy. Future work will extend this framework to entropy--based 
differential operators and to the information--theoretic reconstruction of Riemannian geometry and curvature.

\section*{Acknowledgements}
The author received no funding for this work, and declares no competing interests.

\bibliographystyle{unsrt}
\bibliography{refs.bib}

\end{document}